\documentclass[a4paper,leqno,11pt]{amsart}

\usepackage{amssymb}
\usepackage{amsthm}
\usepackage{tikz-cd}
\usepackage{mathtools}
\usepackage{framed}
\usepackage[mathscr]{eucal}
\usepackage{fullpage}
\usepackage{hyperref}
\usepackage{enumerate}
\usepackage{wasysym}
\usepackage{quiver}
\usepackage{xcolor}
\usepackage{scalerel}
\usepackage{comment}
\usepackage{faktor}
\usepackage[all]{xy}
\hypersetup{colorlinks=true,citecolor=blue,urlcolor =black,linkbordercolor={1 0 0}}
\mathtoolsset{centercolon}
\title{On the topological Brauer group of generalized Kummer varieties}

\author{}
\date{\today}

\def\Z{{\mathbb Z}}

\def\Q{{\mathbb Q}}

\def\Kah{K\"ahler}

\def\hk{hyper-K\"ahler}

\def\hkm{hyper-K\"ahler manifold}

\def\phi{\varphi}

\def\-{\textup{-}}

\def\llra{\hbox to 10mm{\rightarrowfill}}
\def\lllra{\hbox to 15mm{\rightarrowfill}}

\def\llla{\hbox to 10mm{\leftarrowfill}}
\def\lllla{\hbox to 15mm{\leftarrowfill}}

 \def\alt{\mathfrak A}

\DeclareMathOperator{\Bl}{Bl}
\DeclareMathOperator{\Br}{Br}

\DeclareMathOperator{\codim}{codim}

\DeclareMathOperator{\Ext}{Ext}

\DeclareMathOperator{\Kum}{Kum}

\DeclareMathOperator{\NS}{NS}

\DeclareMathOperator{\Supp}{Supp}

\DeclareMathOperator{\Stab}{Stab}

\DeclareMathOperator{\tors}{tors}

\DeclareMathOperator{\Ind}{Ind}
\DeclareMathOperator{\triv}{triv}
 
\def\llra{\hbox to 10mm{\rightarrowfill}}
\def\lllra{\hbox to 15mm{\rightarrowfill}}

\def\subset{\subseteq}

\newtheorem{theorem}{Theorem}[section]
\newtheorem*{theorem*}{Theorem}
\newtheorem{lemma}[theorem]{Lemma}
\newtheorem{proposition}[theorem]{Proposition}

\newtheorem{corollary}[theorem]{Corollary}

\newtheorem*{claim*}{Claim}
\newtheorem{claim}[theorem]{Claim}

\theoremstyle{definition}
\newtheorem{question}[theorem]{Question}

\theoremstyle{remark}

\newtheorem*{remark*}{Remark}
\newtheorem*{note*}{Note}

\def\sss[#1]{{S^{[#1]}}}

\def\setminus{\smallsetminus}

\usepackage[
  backend=biber,
  style=alphabetic,
  sorting=nty,
  maxcitenames=50,
  maxnames=50
]{biblatex}

\renewbibmacro{in:}{}
\AtEveryBibitem{%
  \clearfield{issn} 
  \clearfield{doi} 
  \clearfield{isbn}
  \ifentrytype{online}{}{
    \clearfield{url}
  }
}

\author{Moritz Hartlieb}
\address{Mathematisches Institut, Universität Bonn, Endenicher Allee 60, 53115
Bonn, Germany}
\email{hartlieb@math.uni-bonn.de}

\author{Matteo Verni}
\address{Sorbonne Université, Université Paris Cité, CNRS, IMJ-PRG, F-75005 Paris, France}
 \email{{\tt matteo.verni@imj-prg.fr}}

\begin{document}

\begin{abstract}
We study the topological Brauer group of generalized Kummer varieties. We prove that it vanishes when their dimension is divisible by $4$, while for all other dimensions except  dimension $10$ we prove that it is at most $8$-torsion.
\end{abstract}

\maketitle

\setcounter{tocdepth}{1}
\tableofcontents
\section{Introduction}
The study of \hk \ manifolds is a central topic in the theory of compact Kähler manifolds: in virtue of the Beauville--Bogomolov decomposition theorem \cite[Thm.\ 1]{beauville}, they represent one of three ``fundamental" types of compact Kähler manifolds with \(c_1(K_X)=0\), together with abelian varieties and strict Calabi-Yau manifolds.
Their topology is extremely interesting: on the one hand it is severely restricted compared to that of a general compact \Kah \ manifold, for example by the Fujiki relations, \cite{Fujiki}. Moreover, endowed with the usual Hodge structure, the cohomology of \hkm s largely determines their geometry by the Torelli Theorem, see for example \cite{Markman_torelli}. On the other hand, the topology of \hkm s remains very mysterious, with a lot of basic open questions. One of these questions is the following:

\begin{question}\label{que:hk_torsion}
Let $X$ be a \hk \  manifold. Is $H^*(X, \mathbb Z)$ torsion-free?
\end{question}
For \hk \ manifolds of $K3^{[n]}$-type, the answer is positive by \cite{Markman2007_integral}, which has then been further generalized by \cite{Totaro_2020_integral} to the Hilbert scheme of points on any smooth projective surface with torsion-free cohomology. Already for the next most understood \hk \ deformation type, that of \textit{generalized Kummer varieties} (or \textit{Kummer manifolds}), denoted as \(\Kum_n(A)\), this question is open in its full generality. In \cite{kapfermenet}, however,  Kapfer and Menet give a positive answer in the case of Kummer fourfolds. 

The abelian group $H^3(X, \mathbb Z)_{\tors}$ is of particular interest as it sits in the short exact sequence
$$0 \to H^2(X, \mathbb Z) / \NS(X, \mathbb Z) \otimes_\Z \Q/\Z \to \Br(X) \to H^3(X, \mathbb Z)_{\tors} \to 0.$$
For this reason it is often called the \textit{topological Brauer group}, see also \cite[Sec.\ 1]{grothendieck}, and it is the lowest degree cohomology group whose torsion part is not known for any \hk \ manifold other than those of K3\(^{[n]}\)-type and Kummer fourfolds.

The aim of this note is to prove the following result:
\begin{theorem}\label{thm:main}
If \(n \neq 6\), then
    $$ \gcd(2^3, n) \cdot H^3(\Kum_{n-1}(A), \mathbb Z)_{\tors} = 0.$$
If instead \(n=6\), then
$$6 \cdot H^3(\Kum_{5}(A), \mathbb Z)_{\tors} = 0.$$

\end{theorem}
For odd values of \(n\), this amounts to the following:
\begin{corollary}\label{cor:nice_statement_even}
    The third integral cohomology group of a generalized Kummer variety of dimension divisible by four is torsion-free.
\end{corollary}
As already mentioned, the case \(n=2\) follows from \cite{kapfermenet}. This result confirms folklore expectations, as stated for example in \cite[Rem.\ 2.1.(ii)]{huyb2024}

The strategy of the proof is inspired by and generalizes the approach of \cite{kapfermenet}:
we first consider a rational cover of degree \(n\) by a smooth projective variety with no torsion in degree three cohomology, which forces \(H^3(\Kum_{n-1}(A),\Z)_{\tors}\) to be entirely of \(n\)-torsion. By considering a double cover of a suitable open of \(\Kum_{n-1}(A)\), we then show that \(H^3(\Kum_{n-1}(A),\Z)_{\tors}\) is a $2$-group, via an analysis of the cohomology of the alternating group and the Cartan--Leray spectral sequence. When \(n\) is odd, we obtain Corollary \ref{cor:nice_statement_even}.
We are then left with asking ourselves the following

\begin{question}
    Does \(H^3(\Kum_{2n+1}(A),\Z)_{\tors}\) vanish for \(n \geq 1\) as well?
\end{question}

\subsection{Acknowledgements}
The authors would like to thank Daniel Huybrechts and Emanuele Macrì for their interest in the project, as well as Nick Addington for sharing his computations and suggesting that earlier bounds could be improved. The first author would like to thank Claire Voisin for the invitation to Paris, where this project was initiated. 
Many thanks to Daniel Huybrechts and Paolo Stellari, respectively, for bringing up Question \ref{que:hk_torsion} at the ``K3 surfaces and friends"\footnote{Lorentz Center, Leiden, June 2025  (\href{https://www.lorentzcenter.nl/k3-surfaces-en-friends-brauer-groups-and-moduli.html}{\underline{link}}) 
}
and PRAGMATIC\footnote{University of Catania, Italy, September 2025 (\href{https://www.dmi.unict.it/pragmatic/docs/Pragmatic2025.html}{\underline{link}})} summer schools. We thank the organizers of said events for the wonderful research environment.
Both authors were supported by the ERC Synergy Grant 854361 HyperK.
The first author is grateful for the support provided by the International Max Planck Research School on Moduli Spaces at the Max Planck Institute for Mathematics in Bonn.

\section{Generalized Kummer varieties}
Let us briefly recall the definition of generalized Kummer varieties. Let \(A\) be an abelian surface and let \(A^{[n]}\) denote the Hilbert scheme of \(n\)-points on \(A\). The natural summation map \(\varepsilon \colon A^n\rightarrow A\) induces a morphism from the Hilbert scheme 
\[\Sigma \colon A^{[n]}\to A.\] 
The \textit{generalized Kummer \(2n\)-fold} is defined to be \[\Kum_{n-1}(A) \coloneqq \Sigma^{-1}(0) \hookrightarrow A^{[n]}.\]

We set \(K_{n-1}(A)\coloneqq \varepsilon^{-1}(0) \hookrightarrow A^n\). Of course, we have \(K_{n-1}(A)\simeq A^{n-1}\) as abelian varieties. 
The natural action of \(\mathfrak{S}_n\) on \(A^n\) restricts to \(K_{n-1}(A)\), and \(\Kum_{n-1}(A)\) is a crepant resolution of the quotient variety \(K_{n-1}/\mathfrak{S}_n\).
By \cite[Thm.\ 4]{beauville}, generalized Kummer varieties are \hkm s.

\subsection{\texorpdfstring{\(n\)}{n}-torsion via torsion-freeness of \texorpdfstring{\(H^3(A^{[n-1]}, \mathbb Z)\)}{H3(An, Z)}}

We start with a simple observation:

\begin{lemma}\label{lem:dn_tors}
Let \(X,Y\) be two smooth projective varieties of the same dimension and \(X\dashrightarrow Y\) a dominant rational map of degree \(d\). Let \(N\) be an integer and suppose \(N \cdot H^3(X,\Z)_{\tors}=0\). Then 
\[dN \cdot H^3(Y,\Z)_{\tors}=0.\]
\vspace{-15pt}

\end{lemma}
\begin{proof}
Since \(H^3(-,\Z)_{\tors}\) is a birational invariant, by passing to a resolution of indeterminacies we may assume that \(f\) is everywhere defined. The result then follows from the projection formula.
\end{proof}
\begin{corollary}\label{cor:n-tors_of_Kum}
    The abelian group $H^3(\Kum_{n-1}(A), \mathbb Z)_{\tors}$ is $n$-torsion.
\end{corollary}
\begin{proof}
    The map \(K_{n-1}(A) \dashrightarrow \Kum_{n-1}(A)\) factors through the quotient by \(\mathfrak{S}_{n-1}\) acting on the first \(n-1\) entries of \(A^n\), yielding a map of degree \(n\). 
    The quotient \(K_{n-1}(A)/\mathfrak{S}_{n-1}\) is birational to \(A^{[n-1]}\), hence we have a generically finite rational map  of degree \(n\)
    \[A^{[n-1]} \dashrightarrow \Kum_{n-1}(A).\]   
By \cite{Markman2007_integral} the group $H^3(A^{[n-1]},\Z)$ is torsion-free. The result then follows from Lemma \ref{lem:dn_tors} with \(N=1\).
\end{proof}
\section{Quotient by the alternating group}
Ideally, one would like to apply the argument of the previous section to the rational double cover $$A^{n-1} / \mathfrak A_n \dashrightarrow  \Kum_{n-1}(A).$$ However, there is no smooth projective model \(Z\) of $A^{n-1} / \alt_n$ for which \(H^3(Z, \Z)_{\tors}\) is known. 

Instead, we do the following:
let 
$$\Delta_i \coloneqq \{(a_1, \dots,a_n) \in A^n \mid \#\{a_1, \dots , a_n\}\leq  n-i\} \subset A^n$$
and
$$D_i \coloneqq \{Z \subset A \mid \#\Supp(Z) \leq n-i\} \subset A^{[n]}$$
be the standard filtration of the big diagonal, i.e., \(\Delta_1\), resp.\ the Hilbert--Chow divisor, i.e., \(D_1\). 
We then consider the open subsets \(\Delta \coloneq \Delta_1 \setminus \Delta_2 \subset \Delta_1\) and \(D \coloneqq D_1 \setminus D_2 \subset D_1\).

\begin{lemma}
For any \(k\geq 0\), we have
$$\codim_{K_{n-1}(A)}(\Delta_{k} \cap K_{n-1}(A)) = 2k$$
and
$$\codim_{\Kum_{n-1}(A)}(D_k \cap \Kum_{n-1}(A)) = k.$$
\end{lemma}
\begin{proof}
The first statement is clear. The second statement follows from the description of the fibers of the Hilbert--Chow morphism over the images of $\Delta_k$ in \(A^{(n)}\), cf.\ \cite[III.3]{Briancon1977}.  
\end{proof}

Let $U \coloneqq K_{n-1}(A) \cap (A^n\setminus \Delta_2) \subset A^n$ and $V \coloneqq \Kum_{n-1}(A) \cap (A^{[n]} \setminus D_2) \subset A^{[n]}$.
For codimension reasons, avoiding the smaller diagonals does not change $H^3(-, \mathbb Z)_{\tors}$:
\begin{lemma}\label{lem:torsvkum}
    We have
    $H^3(V, \mathbb Z)_{\tors} \simeq H^3(\Kum_{n-1}(A), \mathbb Z)_{\tors}.$
\end{lemma}
\begin{proof}
We consider the stratification 
\[D_{2}= Z_0 \supset Z_1\supset \cdots \]
where \(Z_{i+1}\) is the singular locus of \(Z_i\). By the long exact sequence of relative cohomology and the Thom isomorphism, we have
\[H^3(\Kum_{n-1}\setminus Z_{1},\Z)\simeq H^3(\Kum_{n-1}\setminus Z_{2},\Z) \simeq \dots \simeq H^3(\Kum_{n-1}(A), \mathbb Z),\]
see \cite[Lemma 11.13]{Voisin_book_I}. By the same argument we also have the exact sequence 
$$ 0\to H^3(\Kum_{n-1} \setminus Z_1,\mathbb Z) \to H^3( \Kum_{n-1}(A) \setminus Z_0, \mathbb Z) \to \mathbb Z^{\pi_0(Z_0 \setminus Z_1)}.$$
Thus 
\[H^3(\Kum_{n-1}(A),\Z)_{\tors}\simeq H^3(\Kum_{n-1}(A)\setminus Z_0,\Z)_{\tors}.\]
Since $V = \Kum_{n-1}(A) \setminus Z_0$, we conclude.
\end{proof}
Since the alternating group acts freely on $U$, the relation between the quotient $U / \alt_n$ and the open subset $V \subset \Kum_{n-1}(A)$ can be made precise as follows:
\begin{lemma}\label{lem:doublecover_V}
There exists a ramified double cover  
    $$\pi \colon \Bl_{\overline{\Delta}} (U/\alt_n)  \to V,$$
    where \(\overline{\Delta}\) is the image in \(U/\alt_n\) of \(\Delta\subset U\).
\end{lemma}
\begin{proof}
By \cite[Sec.\ 8, p.\ 770]{beauville}, we have
\[(\Bl_{\Delta} U )/ \mathfrak S_n \simeq V\]
Moreover, since $\mathfrak A_n$ acts freely on $U$ and stabilizes $\Delta$, we have
$$\Bl_{\Delta}(U) / \mathfrak A_n \simeq \Bl_{\overline{\Delta}}(U / \mathfrak A_n).$$
The action of \(\Z/2\Z\simeq\mathfrak{S}_n/\alt_n\) induced onto the quotient \(\Bl_{\Delta}(U) / \mathfrak A_n \) produces a double cover
 $$\Bl_{\overline\Delta}(U / \mathfrak A_n) \simeq (\Bl_{\Delta} U ) / \mathfrak A_n \to (\Bl_{\Delta} U) / \mathfrak S_n \simeq V,$$
 which is what we wanted.
 \end{proof}

\begin{proposition}\label{prop:spectral_sequence_argument}
    If there are integers $N_p$ for $1 \leq p \leq 3$ such that $$N_p \cdot H^p(\mathfrak A_n, H^{3-p}(K_{n-1}(A), \mathbb Z)) = 0,$$ then we have
    $$2N_1N_2N_3 \cdot H^3(\Kum_{n-1}(A), \mathbb Z)_{\tors} = 0.$$
\end{proposition}
\begin{proof}
Since the action of $\mathfrak A_n$ on $U$ is free, there is the Cartan--Leray spectral sequence
$$H^p(\mathfrak A_n, H^q(U, \mathbb Z)) \Rightarrow H^{p+q}(U / \mathfrak A_n, \mathbb Z).$$
Note that for $p \leq 6$, by the same argument as in Lemma \ref{lem:torsvkum} we have $H^p(U, \mathbb Z) \simeq H^p(K_{n-1}(A), \mathbb Z)$, as the complement of $U \subset K_{n-1}(A)$ is of codimension four. In particular, $H^0(\mathfrak A_n, H^3(U, \mathbb Z)) \simeq H^3(K_{n-1}(A))^{\mathfrak A_n}$ is torsion-free. 
It then follows that $H^3(U / \mathfrak A_n, \mathbb Z)_{\tors}$ is an extension of subquotients of $H^{p}(\mathfrak A_n, H^{3-p}(K_{n-1}(A, \mathbb Z)))$ for $1 \leq p \leq 3$, which implies that  \(N_1 N_2  N_3 \cdot H^3(U / \mathfrak A_n, \mathbb Z)_{\tors} = 0\) .
As blowing-up smooth subvarieties leaves $H^3(-, \mathbb Z)_{\tors}$ invariant, we have
$$H^3(\Bl_{\overline{\Delta}}(U / \mathfrak A_n), \mathbb Z)_{\tors} \simeq H^3(U / \mathfrak A_n, \mathbb Z)_{\tors}.$$
By applying \cite[Thm.\ 5.4]{aguilarprieto} to the double cover of Lemma \ref{lem:doublecover_V} $$\pi \colon \Bl_{\overline\Delta}(U / \mathfrak A_n) \to  V,$$
we have that \(\pi_* \pi^* \alpha =2 \alpha\) for any \(\alpha \in H^*(V,\Z)\), which allows us to conclude by Lemma \ref{lem:torsvkum}.
\end{proof}

\section{Group cohomology}

The aim of this section is to prove the following:
\begin{proposition}\label{prop:spectralval}
    
The cohomology groups \(H^p(\alt_n,H^{3-p}(K_{n-1}(A), \mathbb Z))\) for \(p=1,2,3\) and \(n\geq 3\) are given by
\begin{center}

\begin{tabular}{ |c|c|c|c|c|c|c| } 
 \hline
 $p \backslash n$ & 3 & 4 & 5 & 6 & 7 & $\geq 8$ \\
 \hline
 1 & $0$ & $(\mathbb Z / 4 \mathbb Z)^{4} \oplus (\mathbb Z / 2 \mathbb Z)^6$ & $(\mathbb Z / 2 \mathbb Z)^4 $  &  $(\mathbb Z / 2 \mathbb Z)^4$ & $(\mathbb Z / 2 \mathbb Z)^4$ & $(\mathbb Z / 2 \mathbb Z)^4$\\
 \hline
 2 & $0$ & $(\mathbb Z / 2 \mathbb Z)^4$   & $(\mathbb Z / 3 \mathbb Z)^4$  & $(\mathbb Z / 3 \mathbb Z)^4$  & $ 0 $ & $0$    \\ 
 \hline
 3 & $0$ & $\mathbb Z  / 2 \mathbb Z$ & $\mathbb Z / 2 \mathbb Z$ & $\mathbb  Z/ 6 \mathbb Z$ & $\mathbb Z / 6 \mathbb Z$ & $\mathbb Z / 2 \mathbb Z$ \\ 
 \hline
 \end{tabular}
\end{center}
\vspace{10pt}
In particular, the above groups are \(2\)-torsion for \(n\geq 8\) and \(12\)-torsion for \(n<8\).
\end{proposition}
The values for $n \leq 7$ have been computed using \cite{GAP4}. In what follows, we give the proof for $n \geq 8$. We start by recalling classical computations of group cohomology with trivial \(\Z\) coefficients, which already settles the case \(p=3\) of Proposition \ref{prop:spectralval}.

\begin{proposition}\label{prop:group_coh_an}
    
The cohomology groups \(H^p(\alt_n,\Z_{\triv})\) for \(p=1,2,3\) and \(n\geq 3\) are given by

\begin{center}

\begin{tabular}{ |c|c|c|c|c|c|c| } 
 \hline
 $p \backslash n$ & 3 & 4 & 5 & 6 & 7 & $\geq 8$\\
 \hline
 1 & $0$ & $0$ & $0$ & $0$ & $0$ & $0$\\
 \hline
 2 & $\mathbb Z / 3 \mathbb Z$ & $\mathbb Z / 3 \mathbb Z$ & $0$  & $0$ & $0$ & $0$  \\ 
 \hline
 3 & $0$ & $\mathbb Z  / 2 \mathbb Z$ & $\mathbb Z / 2 \mathbb Z$ & $\mathbb  Z/ 6 \mathbb Z$ & $\mathbb Z / 6 \mathbb Z$ & $\Z/2\Z$\\ 
 \hline
 \end{tabular}
    
\end{center}
\end{proposition}
\begin{proof}

The universal coefficient theorem and the observation that the abelian groups $H_p(\alt_n, \mathbb Z)$ are finite for $p > 0$, c.f.\ \cite[Cor.\ 6.5.10]{weibelhomalg}, yields
\[H^{p}(\mathfrak A_{n}, \mathbb Z_{\triv}) \simeq \Ext^1(H_{p-1}(\mathfrak A_{n}, \mathbb Z_{\triv}), \mathbb Z).
\]
For $p = 2$ the statement follows from $H_1(G, \mathbb Z) = G^{ab}$, while for $p = 3$ we conclude by classical computations going back to Schur \cite{schur}, see for example \cite[Ex.\ 6.9.10]{weibelhomalg}. Note that $\Ext^1(\mathbb Z / n\mathbb Z, \mathbb Z) \simeq \mathbb Z / n \mathbb Z$ for $n > 0$.
\end{proof}

\color{black}
\subsection{Some notation}\label{sec_notation}
To make the ensuing computations more clear, let us introduce some notation.
As above, we consider the short exact sequence
$$0 \to K_{n-1}(A) \to A^n \to A \to 0$$
of abelian varieties with an $\mathfrak A_n$-action induced by the permutation action on $A^n$. We obtain a short exact sequence
$$0 \to H^1(A, \mathbb Z) \to H^1(A^n, \mathbb Z) \to H^1(K_{n-1}(A),\Z) \to 0$$
of $\mathfrak A_n$-modules. Note that the sequence is the tensor product of $H^1(A, \mathbb Z)$ with the short exact sequence
$$0 \to \mathbb Z_{\triv} \to M \to N \to 0,$$
where $M \coloneqq \mathbb Z^n$ is equipped with the permutation action and the map $\mathbb Z_{\triv} \to M$ is the diagonal embedding.

\subsection{Shapiro's Lemma}
As we saw in Proposition \ref{prop:group_coh_an}, we have complete control over the cohomology groups of $\alt_n$ with coefficients in the trivial module $\mathbb Z_{\triv}$. The main tool we are going to use to handle cohomology groups with coefficients in other modules is Shapiro's Lemma. 
We give a brief recollection of the statement before continuing with the proof of Proposition \ref{prop:spectralval}.

Suppose that $H$ is a subgroup of $G$ and $W$ is a left $\mathbb Z[H]$-module. Then, the left $\mathbb Z[G]$-module
$$\Ind_{H}^G(W) \coloneqq \mathbb Z[G] \otimes_{\mathbb Z[H]} W$$
is called the induced $G$-module associated to $H$.
\begin{lemma}[Shapiro's Lemma]\label{lem:shapiro}
    Let $H \subset G$ be a subgroup of finite index and $W$ an $H$-module.
    Then 
    $$H^*(G, \Ind_H^G(W)) \simeq H^*(H, W).$$
\end{lemma}
\begin{proof}
This follows by combining \cite[Lem.\ 6.3.2]{weibelhomalg} and \cite[Lem.\ 6.3.4]{weibelhomalg}.
\end{proof}

For latter use, we recall the following explicit description of induced modules:

\begin{lemma}\label{lem:ind_explicit}
 Let $X$ be a set with a transitive action by a group $G$. Fix $x \in X$ and let $H \coloneqq \Stab_G(x)$ denote the stabilizer of $x$. For each $y \in X$, fix $g_y \in G$ with $g_y \cdot x = y$. Let $W$ be an $H$-module. Then, we have
$$\Ind_H^G(W) \simeq \bigoplus_{y \in X} g_y \otimes W,$$
where $g \in G$ acts on the right hand-side by $g \cdot (g_y \otimes w) = g_{g(y)} \otimes (g_{g(y)}^{-1}gg_y \cdot w)$.
\end{lemma}
\begin{proof}
    Since the action of $G$ on $X$ is transitive, the collection $\{g_y \in \mathbb Z[G] \mid y \in X\}$ forms a basis for the right $H$-module $\mathbb Z[G]$. Thus, we have
    $$\Ind_H^G(W) = \mathbb Z[G] \otimes_{\mathbb Z[H]} W \simeq \bigoplus_{y \in X} g_y \otimes W$$
    with $g(g_y \otimes w) = g_{g(y)} \otimes hw$ for $h = g_{g(y)}^{-1}g g_y \in H$ as claimed.
\end{proof}

If $W = \mathbb Z_{\triv}$ is the trivial $H$-module, the above simplifies as follows:

\begin{corollary}\label{cor:ind_explicit_trivial}
    Let $G$ be a finite group and $X$ a set with a transitive action by $G$. Fix $x \in X$ and let $H \coloneqq \Stab_G(x)$ denote its stabilizer. Then, we have
    $$\Ind^G_H(\mathbb Z_{\triv}) \simeq \mathbb Z[X],$$
    where $\mathbb Z[X]$ is the free $\mathbb Z$-module with basis $\{e_x \mid x \in X\}$ and $G$-action induced by $g(e_x) = e_{g(x)}$. 
\end{corollary}

\subsection{Proof of Proposition \ref{prop:spectralval}: the case \texorpdfstring{${\bf p = 2}$}{(p, q)=(2, 1)} } As the case $p = 3$ was handled by Proposition \ref{prop:group_coh_an}, we proceed with the case $p = 2$.

\begin{lemma}\label{lem:p2} For $n \geq 8,$ we have
 $$H^2(\mathfrak A_{n}, H^1(K_{n-1}(A), \mathbb Z)) = 0.$$
\end{lemma}

\begin{proof}
Keeping the notation introduced in Section \ref{sec_notation}, we have 
$$H^2(\mathfrak A_{n}, H^1(K_{n-1}(A), \mathbb Z)) \simeq H^2(\mathfrak A_n, N) \otimes H^1(A, \mathbb Z).$$ Thus, we may focus on the cohomology group $H^2(\mathfrak A_n, N)$. From the long exact sequence in group cohomology, we obtain the exact sequence
\begin{equation*}\label{eq:Z_M_N} H^2(\mathfrak A_n, M) \to H^2(\mathfrak A_n, N) \to H^3(\mathfrak A_n, \mathbb Z_{\triv}) \to H^3(\mathfrak A_n, M).\end{equation*}
    \begin{claim*}
         We have $H^2(\mathfrak A_n, M) \simeq H^2(\mathfrak A_{n-1}, \mathbb Z_{\triv})$.
    \end{claim*}
    \begin{proof}[Proof of the Claim]
    Consider the natural action of $\mathfrak A_n$ on $X = \{1, \dots, n\}$. By definition, we have $M \simeq \mathbb Z[X]$ as $\mathfrak A_n$-modules. As elements in $X$ are stabilized by groups isomorphic to $\mathfrak A_{n-1}$, we have
$$M \simeq \Ind_{\mathfrak A_{n-1}}^{\mathfrak A_n} \mathbb Z_{\triv},$$
c.f.\ Corollary \ref{cor:ind_explicit_trivial}.
The claim then follows by Shapiro's Lemma \ref{lem:shapiro}.
    \end{proof}
Since $n \geq 8$, we have $H^2(\mathfrak A_{n-1}, \mathbb Z_{\triv}) = 0$ by Proposition \ref{prop:group_coh_an}. It thus remains to show that the map $H^3(\mathfrak A_n, \mathbb Z_{\triv}) \to H^3(\mathfrak A_n, M)$ is injective. 
Composing this map with the isomorphism $H^3(\mathfrak A_n, M) \simeq H^3(\mathfrak A_{n-1}, \mathbb Z_{\triv})$ given by Shapiro's Lemma \ref{lem:shapiro}, we obtain a map $H^3(\mathfrak A_n, \mathbb Z_{\triv}) \to H^3(\mathfrak A_{n-1}, \mathbb Z_{\triv})$, which agrees with the restriction homomorphism, see \cite[p.\ 60]{neukirch}. 
From the classical description of $H^3(\mathfrak A_n, \mathbb Z_{\triv})$ given by Schur (\cite{schur}),  it follows that the restriction homomorphism is injective for $n \geq 7$. We conclude that $H^2(\mathfrak A_n, N)$ vanishes.
\end{proof}

\subsection{Proof of Proposition \ref{prop:spectralval}: the case \texorpdfstring{${\bf p = 1}$}{(p, q)=(2, 1)} }

\begin{lemma}\label{lem:p=1}
    For $n \geq 8$, we have
    $$H^1(\mathfrak A_n, H^2(K_{n-1}(A), \mathbb Z)) \simeq (\mathbb Z / 2 \mathbb Z)^{4}.$$
\end{lemma}

\begin{proof}
Keeping the notation introduced in \ref{sec_notation}, we have
$$H^2(K_{n-1}(A), \mathbb Z) \simeq \wedge^2 H^1(K_{n-1}(A), \mathbb Z) \simeq \wedge^2 (N^{\oplus 4}) \simeq (\wedge^2 N)^{\oplus 4} \oplus (N \otimes N)^{\oplus \binom{4}{2}}$$
as $\mathfrak A_n$-modules.

We first compute $H^1(\mathfrak A_n, \wedge^2 N) \simeq \mathbb Z / 2\mathbb Z$ as follows: There is a short exact sequence
$$0 \to N \otimes_{\mathbb Z} \mathbb Z_{\triv} \to \wedge^2 M \to \wedge^2 N \to 0$$
of $\mathfrak A_n$-modules
and therefore an exact sequence
$$H^1(\mathfrak A_n, N) \to H^1(\mathfrak A_n, \wedge^2M) \to H^1(\mathfrak A_n, \wedge^2N) \to H^2(\mathfrak A_n, N).$$
Using the exact sequence 
$$H^1(\mathfrak A_n, M) \to H^1(\mathfrak A_n, N) \to H^2(\mathfrak A_n, \mathbb Z_{\triv}),$$ combined with the fact that $H^1(\mathfrak A_n, M) \simeq H^1(\mathfrak A_{n-1}, \mathbb Z_{\triv}) = 0$ vanishes by Shapiro's Lemma \ref{lem:shapiro}, and the vanishing of $H^2(\mathfrak A_n, \mathbb Z_{\triv}) = 0$ for $n \geq 5$, cf.\ Proposition \ref{prop:group_coh_an}, we obtain the vanishing of $H^1(\mathfrak A_n, N)$. Moreover, we have $H^2(\mathfrak A_n, N) = 0$ by the proof of Lemma \ref{lem:p2}. Thus, we have $$H^1(\mathfrak A_n, \wedge^2 M) \simeq H^1(\mathfrak A_n, \wedge^2 N),$$ which allows us to conclude by applying Claim \ref{claim:Shapiro's_cool} below.

A similar strategy allows to compute \(H^1(\alt_n, N\otimes N) = 0\): there are short exact sequences
$$0 \to (\mathbb Z_{\triv} \otimes M) + (M \otimes \mathbb Z_{\triv}) \to M \otimes M \to N \otimes N \to 0$$
and
$$0 \to \mathbb Z_{\triv} \otimes \mathbb Z_{\triv} \to (\mathbb Z_{\triv} \otimes M) + (M \otimes \mathbb Z_{\triv}) \to \mathbb Z_{\triv} \otimes N \oplus N \otimes \mathbb Z_{\triv} \to 0$$
of $\mathfrak A_n$-modules and thus exact sequences
$$H^1(\mathfrak A_n, M \otimes M) \to H^1(\mathfrak A_n, N \otimes N) \to H^2(\mathfrak A_n, (\mathbb Z_{\triv} \otimes M) + (M \otimes \mathbb Z_{\triv}))$$
and
$$H^2(\mathfrak A_n, \mathbb Z_{\triv}) \to H^2(\mathfrak A_n, (\mathbb Z_{\triv} \otimes M) + (M \otimes \mathbb Z_{\triv})) \to H^2(\mathfrak A_n, N)^{\oplus 2}.$$
As above, we have $H^2(\mathfrak A_n, \mathbb Z_{\triv}) = 0$ and $H^2(\mathfrak A_n, N) = 0$ by Proposition \ref{prop:group_coh_an} and the proof of Lemma \ref{lem:p2}. Hence, we obtain $$H^2(\mathfrak A_n, (\mathbb Z_{\triv} \otimes M) + (M \otimes \mathbb Z_{\triv})) = 0$$ and conclude by Claim \ref{claim:Shapiro's_cool}.
\end{proof}

\begin{claim}\label{claim:Shapiro's_cool}
       For $n \geq 4$, we have
    \[
    H^1(\alt_n, \wedge ^2 M) \simeq  \Z / 2 \Z \\ \ 
\ \ \text{ and } \ \ \ 
    H^1(\alt_n, M \otimes M) = 0.
    \]
\end{claim}
\begin{proof}
Let $X$ be the set of unordered pairs in $\{1, \dots, n\}$. 
Pick $x_0  =\{1, 2\} \in X$ and let $H \coloneqq \Stab_{A_n}(x_0).$ The stabilizer $H \simeq \mathfrak S_{n-2}$ fits into an extension
$$1 \to \mathfrak A_{n-2} \to H \to \mathbb Z / 2 \mathbb Z \to 1.$$
Let $\mathbb Z_{-}$ denote the $H$-module which is $\mathbb Z$ as an abelian group and on which $\mathfrak A_{n-2}$ acts trivially and the non-trivial element of $\mathbb Z / 2 \mathbb Z$ acts as multiplication by $-1$. As the action of $\mathfrak A_n$ on $\{1, \dots, n\}$ is $2$-transitive for $n \geq 4$, we have
$$\wedge^2 M \simeq \Ind_{H}^{\mathfrak A_n} \mathbb Z_{-},$$
c.f.\ Lemma \ref{lem:ind_explicit}.
By Shapiro's Lemma, there is an isomorphism
$$H^1(\mathfrak A_n, \Ind_H^{\mathfrak A_n} \mathbb Z_-) \simeq H^1(H, \mathbb Z_-).$$
The Hochschild--Serre spectral sequence for the normal subgroup $\mathfrak A_{n-2} \subset H$ yields the exact sequence
$$0 \to H^1(\mathbb Z / 2 \mathbb Z, \mathbb Z_-) \to H^1(H, \mathbb Z_-) \to H^1(\mathfrak A_{n-2}, \mathbb Z_{\triv})^{\mathbb Z / 2 \mathbb Z}.$$
But $\mathfrak A_{n-2}$ is  a finite group and thus $H^1(\mathfrak A_{n-2}, \mathbb Z_{\triv}) = 0$, hence
$$H^1(H, \mathbb Z_-) \simeq H^1(\mathbb Z / 2 \mathbb Z, \mathbb Z_-) \simeq \mathbb Z / 2 \mathbb Z,$$
which yields the first claim.

For the second, let $Y$ denote the set of pairs $\{(i, j) \mid 1 \leq i, j \leq n\}.$ Then we have
$$M \otimes M \simeq \mathbb Z[Y]$$
as an $\mathfrak A_n$-permutation module. As the action of $\mathfrak A_n$ on $\{1, \dots, n\}$ is $2$-transitive for $n \geq 4$, there are exactly two $\mathfrak A_n$-orbits on $Y$: the diagonal $\Delta = \{(i, i)\mid 1 \leq i \leq n\}$ and its complement $Y \setminus \Delta$. Hence, as $\mathfrak A_n$-modules, we have
$$M \otimes M \simeq \mathbb Z[\Delta] \oplus \mathbb Z[Y \setminus \Delta].$$
The stabilizers of elements of $\Delta$ are isomorphic to $\mathfrak A_{n-1}$, while stabilizers of elements of $Y \setminus \Delta$ are isomorphic to $\mathfrak A_{n-2}$. Moreover, we have
$$\mathbb Z[\Delta] \simeq \Ind_{\mathfrak A_{n-1}}^{\mathfrak A_{n}} \mathbb Z_{\triv}  \ \ \ \text{ and } \ \ \  \mathbb Z[Y \setminus \Delta] \simeq \Ind^{\mathfrak A_{n}}_{\mathfrak A_{n-2}} \mathbb Z_{\triv},$$
c.f.\ Corollary \ref{cor:ind_explicit_trivial}.
Applying Shapiro's Lemma \ref{lem:shapiro}, we have
$$H^1(\mathfrak A_n, \mathbb Z[\Delta]) \simeq H^1(\mathfrak A_{n-1}, \mathbb Z_{\triv}) = 0$$
and
$$H^1(\mathfrak A_n, \mathbb Z[Y \setminus \Delta] \simeq H^1(\mathfrak A_{n-2}, \mathbb Z_{\triv}) = 0,$$
as $H^1(-, \mathbb Z_{\triv})$ vanishes for finite groups.
We conclude that $H^1(\mathfrak A_m, M \otimes M) = 0$.
\end{proof}

\section{Conclusion}

In this section, we reap the fruits of our labour and combine the above findings to prove our main result, whose statement we repeat here for convenience.

\begin{theorem*}[Theorem \ref{thm:main}]
If \(n \neq 6\), then
    $$ \gcd(2^3, n) \cdot H^3(\Kum_{n-1}(A), \mathbb Z)_{\tors} = 0.$$
If instead \(n=6\), then
$$6 \cdot H^3(\Kum_{5}(A), \mathbb Z)_{\tors} = 0.$$
\end{theorem*}

\begin{proof}
On the one hand, we have $n \cdot H^3(\Kum_{n-1}(A), \mathbb Z)_{\tors} = 0$ by Corollary \ref{cor:n-tors_of_Kum}. Note that this already implies the result for $n \in \{4, 6, 8\}$. On the other hand, by combining Proposition \ref{prop:spectral_sequence_argument} and Proposition \ref{prop:spectralval}, we have $2^3 \cdot H^3(\Kum_{n-1}(A), \mathbb Z)_{\tors} = 0$ for $n = 3$ or $n \geq 8$ and $2 \cdot 12^3 \cdot H^3(\Kum_{n-1}(A), \mathbb Z)_{\tors} = 0$ for $n=5,7$. The claim then follows by taking the greatest common divisor of $n$ and $2^3$ (respectively \ $2 \cdot 12^3$).
\end{proof}

We conclude by highlighting the following corollary once more:
\begin{corollary}
   There is no torsion in the third cohomology group of \hkm s of $\Kum_{2n}$-type.
\end{corollary}

\printbibliography

\end{document}